\documentclass[a4paper,11pt]{amsart}

\usepackage{amsmath,amssymb}
\usepackage{amsthm}
\usepackage{mathrsfs}
\usepackage{enumitem}
\usepackage{mathtools}
\usepackage{longtable,filecontents,ltxtable}
\usepackage[all]{xy}

\theoremstyle{plain}
\newtheorem{theorem}{Theorem}[section]
\newtheorem{proposition}[theorem]{Proposition}
\newtheorem{lemma}[theorem]{Lemma}
\newtheorem{corollary}[theorem]{Corollary}

\theoremstyle{definition}
\newtheorem{definition}[theorem]{Definition}
\newtheorem{conjecture}[theorem]{Conjecture}
\newtheorem*{acknowledgements}{Acknowledgements}

\theoremstyle{remark}
\newtheorem{remark}[theorem]{Remark}
\newtheorem{notation}[theorem]{Notation}
\newtheorem{convention}[theorem]{Convention}

\numberwithin{equation}{theorem}

\DeclareMathOperator{\Pic}{Pic}
\DeclareMathOperator{\NE}{NE}
\DeclareMathOperator{\Proj}{Proj}
\DeclareMathOperator{\pr}{pr}

\newcommand{\nequiv}{\equiv _\mathrm{num}}

\newcommand{\sE}{\mathscr{E}}
\newcommand{\sF}{\mathscr{F}}

\newcommand{\sL}{\mathscr{L}}

\newcommand{\sC}{\mathscr{C}}
\newcommand{\sS}{\mathscr{S}}
\newcommand{\sN}{\mathscr{N}}

\newcommand{\cO}{\mathcal{O}}

\newcommand{\bC}{\mathbb{C}}

\newcommand{\bP}{\mathbb{P}}
\newcommand{\bQ}{\mathbb{Q}}

\newcommand{\bZ}{\mathbb{Z}}

\newcommand{\pbP}[2]{\left(\bP^{#1}\right)^{#2}}
\newcommand{\pbQ}[2]{\left(\bQ^{#1}\right)^{#2}}
\newcommand{\pbPB}[2]{\big(\bP(#1)\big)^{#2}}
\AtBeginDocument{%
   \def\MR#1{}
}

\setenumerate{label=(\arabic*),
font=\normalfont
}

\title[CP $n$-folds with $\rho > n-5$]{Fano $n$-folds with nef tangent bundle and Picard number greater than $n-5$}
\author[A. KANEMITSU]{Akihiro KANEMITSU}
\date{}
\address{Graduate School of Mathematical Sciences\\The University of Tokyo\\3-8-1 Komaba\\Meguro-ku, Tokyo 153-8914, Japan}
\email{kanemitu@ms.u-tokyo.ac.jp}

\begin{document}

\begin{abstract} 
We prove that Fano $n$-folds with nef tangent bundle and Picard number greater than $n-5$ are rational homogeneous manifolds.
\end{abstract}

\maketitle

%%%%%%%%%%%%%%%%%%%%%%%%%%%%%%%%%%%%
%%%%%%%%%%%%%%%%%%%%%%%%%%%%%%%%%%%%

\section*{Introduction}
In this paper, we continue our study of the following conjecture: 

\begin{conjecture}[Campana-Peternell conjecture \cite{CP1}]
\label{CP}
Any Fano manifold $X$ with nef tangent bundle is a rational homogeneous manifold.
\end{conjecture}

In the previous work~\cite{Ka}, the author proved that Conjecture~\ref{CP} is true in dimension five, via the results of \cite{CMSB,Hw,Mi,Mok,W2}.
Hence Conjecture~\ref{CP} is known to be true in dimension at most five \cite{CP1,CP2,Hw,Mok}.
For further results or background materials about Conjecture~\ref{CP}, we refer the reader to the survey article \cite{MOSWW}.

\medskip
In \cite{CP1,CP2,W2}, Conjecture~\ref{CP} in the case of Picard number greater than one are treated by inductive approaches.
More precisely, for each $n=3$, $4$ or $5$, they proved Conjecture~\ref{CP} for $n$-folds with Picard number greater than one using the results in dimension at most $n-1$. 

Our objective is to prove Conjecture~\ref{CP} for $6$-folds with Picard number $\rho _{X}>1$ using the above results in dimension at most five.
In fact, we obtain a more general result as follows:

\begin{theorem}\label{rhon-4}
Conjecture~\ref{CP} is true in dimension $n$ with Picard number $\rho _X > n-5$. 
\end{theorem}

The same result in the case where $\rho _{X} > n-4$ is independently obtained by K.~Watanabe \cite{W3}.
The idea of Watanabe's proof and ours in the case where the manifold in question has large dimension are essentially the same; the idea is to use results of R.~Mu{\~n}oz, G.~Occhetta, L.E.~Sol{\'a}~Conde, K.~Watanabe and J.A.~Wi{\'s}niewski \cite{MOSW,OSWW}.

We explain the idea in more details:
A Fano manifold is called a \emph{CP manifold} if the tangent bundle is nef.
Given a CP $n$-fold $X$ ($n\geq 6$) with Picard number $\rho_{X} > n-5$, we have a non-trivial contraction $f \colon X \to Y$.
By the work of J.-P.~Demailly, T.~Peternell and M.~Schneider \cite{DPS}, the contraction $f$ is smooth, and hence, the fibers and the target $Y$ are CP manifolds (see Proposition~\ref{cont} below).
Furthermore, if the dimension of $X$ is large enough, we can show that they have a contraction onto a CP manifold $M$ whose elementary contractions are smooth $\bP^{1}$-fibrations (in \cite{MOSW}, such a manifold $M$ is called an \emph{FT manifold}, see Definition~\ref{FT}).
Then, by results of R.~Mu{\~n}oz, G.~Occhetta, L.E.~Sol{\'a}~Conde, K.~Watanabe and J.A.~Wi{\'s}niewski, $M$ is a complete flag variety and $X \simeq F \times M$, where $F$ is a fiber of the contraction.
Hence we can prove Theorem~\ref{rhon-4} by an inductive approach.

On the other hand, in lower dimensional cases, a CP manifold does not admit a contraction onto an FT manifold in general.
Hence we need to treat them by case-by-case argument.
The main difficult part is to prove that Conjecture~\ref{CP} is true for  CP $6$-folds $X$ with $\bP ^{r}$-bundle structure.
This is essentially done in Section~\ref{CPbundles}.

In Section~2, following the notion of Fano bundle, we introduce the notion of \emph{CP bundle}.
A vector bundle $\sE$  is said to be a \emph{CP bundle} if $\bP (\sE)$ is a CP manifold.
Given a Fano or CP bundle, one can define an invariant $\tau$ called \emph{slope} (see Definition~\ref{slope}).
In \cite{MOS}, severe restrictions on the pair $(Y,\sE)$ are obtained by the numerical conditions on $\tau$ (see, for instance, \cite[the proof of Proposition~4.4]{MOS}).
Also in the present paper, the numerical conditions on slopes play an important role.

\medskip
For our purpose, we rephrase the statement of Theorem~\ref{rhon-4} in the following form.

\begin{theorem}\label{CPk=4}
Let $X$ be a CP $n$-fold with Picard number $\rho _X > n-5$,
then $X$ is one of the following:

\medskip
\noindent
\begin{longtable}{|l|X|}
 \hline
 $\rho_{X}$ & $X$\\
 \hline
 
 $n-4$ & $\pbP{1}{n-5}  \times \big[ \bP ^5$, $\bQ ^5$ or  $K(G_2) \big]$,\\
       & $\pbP{1}{n-6}  \times \big[ \bP (\sC )$, $\bP ^2 \times \bP ^4$, $\bP ^2 \times \bQ ^4$, $\pbP{3}{2}$, $\bP ^3 \times \bQ ^3$ or $\pbQ{3}{2} \big]$,\\
       & $\pbP{1}{n-7}  \times \big[ \pbP{2}{2} \times \bP ^3$, $\pbP{2}{2} \times \bQ ^3$, $\bP ^2 \times \bP (\sS _i)$, $\bP ^2 \times \bP (T_{\bP ^3} )$, $\bP (\sN ) \times \bP ^3$ or $\bP (\sN ) \times \bQ ^3 \big]$,\\
       & $\pbP{1}{n-7}  \times \bP (T_{\bP ^2}) \times     \big[ \bP ^4$ or $\bQ ^4 \big]$,\\
       & $\pbP{1}{n-8}  \times \big[ \pbPB{\sN}{2}$, $\bP(\sN) \times \pbP{2}{2}$, $\pbP{2}{4}$ or $\bP ^2 \times F(1,2,3;4) \big]$,\\
       & $\pbP{1}{n-8}  \times \bP (T_{\bP ^2}) \times     \big[\bP (\sS _i)$, $\bP(T_{\bP ^3})$, $\bP ^2 \times \bP ^3$ or $\bP ^2 \times \bQ ^3\big]$,\\
       & $\pbP{1}{n-9}  \times \bP (T_{\bP ^2}) \times     \big[\pbP{2}{3}$, $\bP ^2 \times \bP(\sN)$ or $F(1,2,3;4) \big]$,\\
       & $\pbP{1}{n-9}  \times \pbPB{T_{\bP ^2}}{2} \times \big[  \bP ^3$ or   $\bQ ^3 \big]$,\\
       & $\pbP{1}{n-10} \times \pbPB{T_{\bP ^2}}{2} \times \big[\bP (\sN)$ or $\pbP{2}{2} \big]$,\\
       & $\pbP{1}{n-11} \times \pbPB{T_{\bP ^2}}{3} \times \bP ^2$,\\
       & $\pbP{1}{n-12} \times \pbPB{T_{\bP ^2}}{4}$\\
 \hline
 
 $n-3$ & $\pbP{1}{n-4}  \times \big[\bP ^4$ or $ \bQ ^4 \big]$,\\
       & $\pbP{1}{n-5}  \times \big[\bP (\sS _i)$, $ \bP (T_{\bP ^3} )$, $ \bP ^2 \times \bP ^3$ or $ \bP ^2 \times \bQ ^3 \big]$,\\
       & $\pbP{1}{n-6}  \times \big[\pbP{2}{3}$, $\bP ^2 \times \bP (\sN )$ or $ F(1,2,3;4)\big]$,\\
       & $\pbP{1}{n-6}  \times \bP (T_{\bP ^2}) \times \big[\bP ^3$ or $ \bQ ^3\big] $,\\
       & $\pbP{1}{n-7}  \times \bP (T_{\bP ^2}) \times \big[\bP (\sN )$ or $\pbP{2}{2} \big]$,\\
       & $\pbP{1}{n-8}  \times \pbPB{T_{\bP ^2}}{2} \times \bP ^2$,\\
       & $\pbP{1}{n-9}  \times \pbPB{T_{\bP ^2}}{3}$\\
 \hline 
 
 $n-2$ & $\pbP{1}{n-3}  \times \big[\bP ^3$ or $\bQ ^3 \big]$,\\
       & $\pbP{1}{n-4}  \times \big[\bP (\sN )$ or $ \pbP{2}{2} \big]$,\\
       & $\pbP{1}{n-5}  \times \bP (T_{\bP ^2}) \times\bP ^2 $,\\
       & $\pbP{1}{n-6}  \times \pbPB{T_{\bP ^2}}{2}$ \\
 \hline

 $n-1$ & $\pbP{1}{n-2}  \times \bP ^2$,\\
       & $\pbP{1}{n-3}  \times \bP (T_{\bP ^2})$ \\
 \hline

    $n$ & $\pbP{1}n$ \\
 \hline
\end{longtable}

\medskip
\noindent
where $\sN $ is the null-correlation bundle on $\bP ^3$, $\sS _i$ $(i=1,2)$ the spinor bundles on $\bQ ^4$, $\sC $ the Cayley bundle on $\bQ ^5$ and $F(1,2,3;4)$ the variety of all complete flags in $\bC^4$.
\end{theorem}

For the definition of the null-correlation bundle $\sN $,
the spinor bundles $\sS _i$ and the Cayley bundle $\sC $, we refer the reader to \cite{OSS,O1,O2}.
Note that the null-correlation bundle $\sN$ is a vector bundle of rank $2$ over $\bP^{3}$ and that the projectivization $\bP (\sN)$ of the null-correlation bundle is isomorphic to the projectivization of the spinor bundle on $\bQ ^3$.

\medskip
As a byproduct of the proof of Theorem~\ref{rhon-4}, we obtained the following:

\begin{theorem}[=Theorem~\ref{2rho}]
Let $X$ be a CP $n$-fold which does not admit a contraction onto an FT manifold.
Then  $n \geq 2\rho_{X}$.
Furthermore, the following hold:

\begin{enumerate}
\item
If the equality holds, then $X \simeq \pbP{2}{\rho_X}$.
\item
If $ n = 2\rho_X+1$, then
$X \simeq \pbP{2}{\rho_X -1} \times \bP ^3$,
$\pbP{2}{\rho_X -1} \times \bQ ^3$,
$\pbP{2}{\rho_X -2} \times \bP (\sS _i)$ or
$\pbP{2}{\rho_X -2} \times \bP (T_{\bP ^3} )$.
\end{enumerate}
\end{theorem}

\begin{convention}
We will work in the category of complex projective varieties unless otherwise stated.
Given a vector bundle $\sE$ over a variety, we will denote by $\bP (\sE)$ the associated projective space bundle of one-dimensional quotients, that is,  $\bP(\sE) \coloneqq \Proj \big(S(\sE)\big)$.

A \emph{smooth $\bP ^r$-fibration} is a smooth morphism whose  fibers are isomorphic to $\bP^r$.
On the other hand, a \emph{$\bP^r$-bundle} is the projectivization of a vector bundle. 

Unless otherwise stated, $X$ is a CP manifold,
$n$ is the dimension of $X$ and
$\rho _X$ is the Picard number of $X$.
\end{convention}

\begin{acknowledgements}
The author wishes to express his gratitude to his supervisor Professor Yoichi Miyaoka for his encouragement, comments and suggestions.
The author is also grateful to Professor Kiwamu Watanabe for his helpful comments and suggestions, and for sending his preprint \cite{W3}.
The author also wishes to thank Professor Hiromichi Takagi and Doctor Fumiaki Suzuki for their careful reading of the manuscript and for helpful comments.
The author is a JSPS Research Fellow and he is supported by the Grant-in-Aid for JSPS fellows (JSPS KAKENHI Grant Number 15J07608).
This work was supported by the Program for Leading Graduate Schools, MEXT, Japan.
\end{acknowledgements}

\section{Preliminaries}\label{pre}
In this section, we collect some results which we use later.

\subsection{CP manifolds}

\begin{definition}[{\cite[Definition~1.4]{MOSWW}}]
A Fano manifold $X$ is said to be a \textit{CP manifold} if the tangent bundle of $X$ is nef.
\end{definition}

CP manifolds with dimension at most five are classified:

\begin{theorem}[{\cite{CP1,CP2,CMSB,Hw,Ka,Mi,Mok,W2}}]  \label{CP5}
Let $X$ be a CP manifold of dimension at most five. Then $X$ is a rational homogeneous manifold.
\end{theorem}
In this case, the explicit form of $X$ as in the table of Theorem~\ref{CPk=4} is also known.

\subsection{Contractions of CP manifolds}
Contractions of CP manifolds are similar to those of rational homogeneous manifolds: 

\begin{proposition}
[{\cite[Proposition~4]{MOSW}}]
\label{cont}
Let $X$ be a CP manifold and $\pi  \colon X \to Y$ a contraction.
Then the following properties hold:

\begin{enumerate}
\item The morphism $\pi$ and $Y$ are smooth.
In particular, the fibers and $Y$ are CP manifolds.\label{cont1}

\item $\rho_{X} \leq \dim X$ \label{cont2}

\item The Picard number of a $\pi$-fiber $F$ is $\rho _X - \rho _Y$ and $j_* \big(\NE (F)\big) = \NE (X) \cap j_* \big( N _1 (F) \big)$, where $j\colon F \to X$ is the inclusion. \label{cont3}

\item $\NE (X)$ is simplicial. \label{cont4}
\end{enumerate}

\end{proposition}

\begin{proof}
\ref{cont1} The morphism $\pi$ and $Y$ are smooth by \cite[Theorem~5.2]{DPS} or \cite[Theorem~4.4]{SW}.
It follows from \cite[Corollary~2.9]{KMM3} that $Y$ is a Fano manifold.
Fibers are also Fano by adjunction.
Furthermore, by \cite[Proposition~2.11]{CP1} the tangent bundles of fibers and $Y$ are nef.
Hence they are CP manifolds.

\ref{cont2} This is a consequence of \ref{cont1}.

For proofs of \ref{cont3} and \ref{cont4}, we refer the reader to \cite[Proposition~4]{MOSW}.
\end{proof}

\subsection{Characterization of complete flag manifolds}
Recently, in \cite{OSWW}, a characterization of complete flag manifolds was obtained by G.~Occhetta, L.E.~Sol{\'a}~Conde, K.~Watanabe and J.A.~Wi{\'s}niewski.
We briefly recall the results of \cite{OSWW}.
For details, we refer the reader to \cite{MOSW}, \cite{OSWW}.

\begin{definition}[{\cite[Definition~1]{MOSW}}] \label{FT}
A CP manifold $M$ is said to be an {\it FT manifold} if every elementary contraction of $M$ is a smooth $\bP ^1$-fibration.
\end{definition}

\begin{theorem}[{\cite[Theorem~1.2]{OSWW}}] \label{CPFT}
Let $M$ be a Fano manifold.
Assume  that every elementary contraction of $M$ is a smooth $\bP ^1$-fibration.
Then $M$ is a complete  flag variety, that is, $M \simeq G/B$ where $G$ is a semisimple group and $B$ is a Borel subgroup.
\end{theorem}
Hence, in particular, FT manifolds are complete flag varieties.
Furthermore, in \cite{MOSW}, the following property of FT manifolds is proved:
\begin{proposition}[{\cite[Proposition~5]{MOSW}}] \label{iniFT}
Let $X$ be a CP manifold.
If there exists a contraction $\pi\colon  X \to M$ onto an FT manifold $M$, then $X \simeq F \times M$ and $\pi$ is the second projection, where $F$ is a fiber of $\pi$.
\end{proposition}

As a corollary, we have:
\begin{corollary}\label{FTdom}
Let $n>0$ and $k \geq 0$ be integers.
Assume that Conjecture~\ref{CP} is true for any CP manifold $Y$ with $\dim Y < n$ and $\dim Y - \rho _{Y} \leq k$.

Then Conjecture~\ref{CP} is true for CP $n$-fold $X$ with $n-\rho _X \leq k$ which admits a contraction $f \colon X \to M$ onto an FT manifold $M$.
\end{corollary}
\begin{proof}
By Proposition~\ref{iniFT}, we have $X \simeq F \times M$, where $F$ is the fiber of $f$.
By Theorem~\ref{CPFT}, $M$ is a complete flag variety.

By Proposition~\ref{cont}, $F$ is a CP manifold with dimension $n- \dim M$ and Picard number $\rho_X - \rho _M$.
Then we have $\dim F -\rho_F \leq n- \rho _X  \leq k$ since $\rho_M \leq \dim M$.
Hence $F$ is a rational homogeneous manifold by our assumption, and the assertion follows.
\end{proof}

\section{CP Bundles}\label{CPbundles}
A vector bundle $\sE$ on a manifold $Y$ is called a \emph{Fano bundle} if the projectivization $\bP (\sE)$ is a Fano manifold, and they are classified in several cases.
For more details and results, we refer the reader to \cite{MOS} and the reference therein.

\begin{definition}
A vector bundle $\sE$ on a manifold $Y$ is said to be a \emph{CP bundle} if the projectivization $\bP (\sE)$ is a CP manifold.
\end{definition}

\begin{remark}\label{rem_bundle}
\hfill
\begin{enumerate}
 \item If $\sE$ is a CP bundle over $Y$, then $Y$ is a CP manifold by Proposition~\ref{cont}. 
 \item If $Y$ is a rational manifold or a curve, then the Brauer group of $Y$ is trivial and hence all smooth $\bP^{r-1}$-fibration over $Y$ is a $\mathbb{\bP}^{r-1}$-bundle (see, e.g.\ \cite[Proposition~2.5]{W2}). \label{rem_bundle2}

\end{enumerate}
\end{remark}

\subsection{Triviality of CP bundles}
We prove several characterizations of triviality of CP bundles.
\begin{proposition}\label{trivb}
Let $\sE$ be a CP bundle of rank $r$ over a manifold $Y$ and $\pi \colon  \bP (\sE) \to Y$ the natural projection.
Then the following are equivalent:
\begin{enumerate}
 \item $\bP(\sE)$ is trivial.\label{trivb1}
 \item The relative anticanonical divisor $-K_{\pi}$ is nef.\label{trivb2}
 \item For every rational curve $f \colon \bP ^1 \to Y$, the base change of $\pi$ by $f$ is trivial.\label{trivb3}
 \item For every rational curve $f \colon \bP ^1 \to Y$ whose image generates an extremal ray, the base change of $\pi$ by $f$ is trivial.\label{trivb4}
 \item For every elementary contraction $f \colon Y \to Z$ and every fiber $F$ of $f$, the base change of $\pi$ over $F$ is trivial.\label{trivb5}
 \item $\sE$ splits into a direct sum of line bundles.\label{trivb6}
 \item $\sE \simeq \sL ^{\oplus r}$ for a line bundle $\sL$.\label{trivb7}
\end{enumerate}
\end{proposition}

\begin{proof} 
Set $X \coloneqq \bP(\sE)$.

The implications \ref{trivb1} $\Rightarrow$ \ref{trivb2} $\Rightarrow$ \ref{trivb3} $\Rightarrow$ \ref{trivb4} and \ref{trivb1} $\Rightarrow$ \ref{trivb7} $\Rightarrow$ \ref{trivb6} are obvious.

\ref{trivb4} $\Rightarrow$ \ref{trivb5} and \ref{trivb3} $\Rightarrow$ \ref{trivb1}.
By Proposition~\ref{cont}, the fiber $F$ and $Y$ are Fano manifolds.
Hence they are rationally connected \cite[Theorem~0.1]{KMM3}.
The assertion follows from \cite[Proposition~2.4]{MOS}.

\ref{trivb5} $\Rightarrow$ \ref{trivb2}.
It is enough to see that $-K_{\pi}$ is nef on every extremal ray of $\NE(X)$.
Let $R$ be the ray corresponding to $\pi$.
Obviously $-K_{\pi}$ is nef on $R$.
On the other hand, by Proposition~\ref{cont}, there is a one-to-one correspondence between the set of rays in $\NE(X)$ which are not $R$ and the set of rays of $\NE(Y)$.
Hence the assertion follows.

We already see that the first five of the conditions are equivalent.
Hence we may assume that $Y$ has Picard number one.

\ref{trivb6} $\Rightarrow$ \ref{trivb1}.
Since $\Pic (Y) \simeq \bZ$, we can write $\sE \simeq \bigoplus \cO_Y (a_i)$ with integers $a_1 \geq \dotsb \geq a_{r}$, where $\cO _Y(1)$ is the generator of $\Pic (Y)$.
By twisting with a line bundle, we may assume that $a_1 = 0$.
Then, there is a splitting exact sequence:
\[
0 \longrightarrow \cO_Y \longrightarrow \sE \simeq \bigoplus \cO_Y (a_i) \longrightarrow \sF \simeq \bigoplus^{}_{i>1} \cO_Y (a_i) \longrightarrow 0.
\]
This gives a projective subbundle $\bP(\sF) \subset \bP(\sE)$ with normal bundle $\cO(1)$ of $\bP(\sF)$.
Since $\bP(\sE)$ is a CP manifold, the normal bundle is nef, and hence, $a_i = 0$ for all $i$.
\end{proof}

\subsection{The slope of CP bundle}
In the rest of this section, we assume that the Picard number of $Y$ is one.
First, we fix some notation.
\begin{notation}
Let $\sE$ be a CP bundle of rank $r$ over a manifold $Y$ with Picard number one, and set $X \coloneqq \bP (\sE)$ with the natural projection $\pi \colon X \to Y $.
We will denote by $\xi$ the class of the tautological line bundle $\cO_{\bP (\sE)} (1)$ and by $H$ the pullback of the ample generator of $\Pic Y$.

Then $-K_\pi = r\xi - \pi ^* c_1(\sE )$ and $-K_X = r\xi + \pi ^*\bigl(-K_Y - c_1 (\sE )\bigr)$.
\end{notation}

In \cite{MOS}, the following invariant is introduced:

\begin{definition}[{\cite[Definition~2.1]{MOS}}]\label{slope}
Let the notation be as above.
The \emph{slope} of the pair $\left(Y,\sE \right)$ is the real number $\tau$ such that $-K_\pi + \tau H$ is nef but not ample.
\end{definition}

\begin{proposition}[{\cite[Remark~2.9]{MOS}}] \label{tau}
The following hold:
\begin{enumerate}
\item $0\leq \tau < F_Y$, where $F_Y$ is the Fano index of $Y$. \label{tau1}
\item $\tau \in \bQ $. \label{tau2}
\item $-K_\pi+ \tau H$ is semiample. \label{tau3}
\end{enumerate}
\end{proposition}
\begin{proof}
\ref{tau1} Since $X$ is Fano, we have $\tau < F_Y$.
The inequality $0\leq \tau$ follows from \cite[Corollary~2.8]{KMM3}.

The second and third parts follow from the rationality theorem and the Kawamata-Shokurov base point free theorem (\cite{KMM} or \cite{KM}).
\end{proof}

As a corollary of Proposition~\ref{trivb}, we have the following:
\begin{corollary}
\label{tau0}
If $\tau = 0$, then $X \simeq \bP ^{r-1} \times Y$.
\end{corollary}

\subsection{CP $6$-folds which admit projective space bundle structures}
We restrict our attention  to pairs $(Y,\sE )$ with $\dim \bP (\sE) =6$ and the Picard number of $Y$ is one.

Using the classification of Fano bundles of rank two, we have the following:

\begin{proposition}\label{P1b}
Let $\sE$ be a CP bundle of rank two over $Y \simeq \bP ^5$, $\bQ ^5$ or $K(G_2)$.
Then $\bP (\sE )$ is a rational homogeneous manifold.
In particular, $X \simeq \bP ^1 \times \bP ^5$, $\bP ^1 \times \bQ ^5$, $\bP ^1 \times K(G_2)$ or $\bP (\sC)$, where $\sC$ is the Cayley bundle on $\bQ ^5$. 
\end{proposition}

\begin{proof}
By Proposition~\ref{trivb} \ref{trivb6} $\Rightarrow$ \ref{trivb1} , we may assume that $\sE $ is indecomposable.
If $Y \simeq \bP ^5$ or $\bQ ^5$, then the assertion follows from~\cite[Main Theorem 2.4]{APW}.
If $Y \simeq K(G_2)$, then another elementary contraction of $X$ is a $\bP ^1$-bundle by Proposition~\ref{cont} and \cite[Lemma~6.1]{MOS}.
The assertion follows from \cite[Theorem~6.5]{MOS} or \cite[Theorem~1.1]{W1}.
\end{proof}

The main result of Section~2 is the following:
\begin{theorem} \label{triv}
Let $\sE$ be a CP bundle of rank $3$ on $\bP^4$ or $\bQ^4$.
Then $\bP (\sE) \simeq \bP ^2\times\bP ^4$ or $\bP ^2\times\bQ ^4$. 
\end{theorem}

Note that, by the definition of Chern classes, the following holds: 
\begin{align} \label{xi} 
\sum^{r}_{i=0} \left( -1 \right) ^i  \pi ^* c_i(\sE )\cdot \xi^{r-i} =0.
\end{align}

Since $X$ is a CP manifold with Picard number $\rho _{X} =2$, we have another elementary contraction $p\colon  X \to Z$. Note that $Z$ is a CP manifold and $p$ is a smooth morphism by Proposition~\ref{cont}.
Furthermore, we have $\dim Z \leq 4$. Otherwise, $X$ is a $\bP ^1$-bundle over rational homogeneous 5-fold $Z$. This contradicts Proposition~\ref{P1b}.
\begin{proof}[Proof of Theorem~\ref{triv} for the case $Y \simeq \bP ^4$]
\hfill

If $Y \simeq \bP ^4 $, then the hyperplane section $H_Y$ generates the Chow ring of $Y$.
We identify the $i$-th Chern classes $c_i(\sE )$ with an integer $c_i$.
We define two invariants:
\begin{align*}
\varDelta &\coloneqq {c_1}^2-3c_2, \\
\gamma &\coloneqq 2 {c_1}^3 - 9 c_1 c_2 + 27 c_3.
\end{align*}

By \eqref{xi} and direct computations, we have the following:

\begin{enumerate}
\item
$\left(-K_\pi \right)^3=9\,\varDelta  H^2\cdot\xi - \left(3c_1 \varDelta - \gamma \right) H^3$,
\item
$\left(-K_\pi \right) \cdot H^5=0$,\,
$\left(-K_\pi \right)^2 \cdot H^4=9$,\,
$\left(-K_\pi \right)^3 \cdot H^3=0$,\,
$\left(-K_\pi \right)^4 \cdot H^2=27\varDelta$,\,
$\left(-K_\pi \right)^5 \cdot H= 9 \gamma$,\,
$\left(-K_\pi \right)^6=81\varDelta ^2$.
\end{enumerate}

Hence,
\begin{align}
&\left( -K_\pi+ \tau H \right) ^5 \cdot H = 9 \left( 10\tau ^3 + 15\varDelta \tau + \gamma \right) \label{a=0},\\
&\left( -K_\pi+ \tau H \right) ^6 = 27 \left( 5\tau ^4 + 15\varDelta \tau ^2 + 2\gamma\tau + 3\varDelta ^2 \right). \label{b=0}
\end{align}

Since $\dim Z \leq 4$, two equalities
$\left( -K_\pi + \tau H \right)^5 \cdot H=0$ and 
$\left( -K_\pi + \tau H \right)^6=0$ hold.
By using \eqref{a=0} and \eqref{b=0}, we have
\begin{eqnarray*}
\varDelta=\frac{5\tau ^2 \pm 3\tau ^2 \sqrt{\mathstrut 5}}{2}.
\end{eqnarray*}
Hence $\tau = 0$, i.e.\ $-K_\pi$ is nef.
Therefore $X$ is isomorphic to $\bP ^2\times\bP ^4$ by Corollary~\ref{tau0}.
\end{proof}

\begin{proof}[Proof of Theorem~\ref{triv} for the case $Y \simeq \bQ ^4$]
\hfill

If $Y \simeq \bQ ^4 $, then the hyperplane section $H_Y$ and two planes $P_{1,Y}, P_{2,Y}$ generate the Chow ring of $Y$.
We will denote by $L_Y$ the class of a line on $\bQ ^4$.
The intersection products of them are as follows:
$H_{Y}^{2}=P_{1,Y}+P_{2,Y}$, $H_{Y} \cdot P_{i,Y}=L_{Y}$, $H_{Y} \cdot L_{Y}=1$, $P_{i,Y}^{2}=1$ and $P_{1,Y}\cdot P_{2,Y}=0$.
We identify the $i$-th Chern classes $c_i(\sE )$ with an integer $c_i$, except for $c_2(\sE )$ which we identify with a pair of integer $(a,b)$.
We will denote by $H$, $P_i$ and $L$ the pullback of $H_Y$,$P_{i,Y}$ and $L_Y$ on $X$.
By twisting $\sE$ with a line bundle, we may assume that $c_1=1$, $2$ or $3$.

Similarly to the case $Y \simeq \bP ^4$,  We define some invariants:
\begin{align*}
c_2 &\coloneqq c_{2}(\sE) \cdot H_{Y}^{2}=a+b,\\
\varDelta (\sE ) &\coloneqq c_{1}(\sE)^{2}-3c_{2}(\sE)={c_1}^2 H_{Y}^2 - 3 \left( aP_{1,Y} + bP_{2,Y} \right),\\
\delta &\coloneqq \varDelta(\sE) \cdot H_{Y}^{2}=2{c_1}^2-3c_2,\\
\gamma &\coloneqq 4 {c_1}^3 - 9 c_1 c_2 + 27 c_3,\\
\beta &\coloneqq \varDelta(\sE)^{2}= \dfrac{\delta ^2 + 9\left(a-b\right)^2}{2}.\\
\end{align*}

Then, by \eqref{xi} and direct computations, we have
\begin{enumerate}
\item
$\left(-K_\pi\right)^3=9\, \pi^{*}\varDelta(\sE) \cdot \xi - \left(3c_1 \delta - \gamma \right) L$.
\item
$\left(-K_\pi\right)\cdot H^5=0$,\,
 $\left(-K_\pi\right)^2\cdot H^4=18$,\,
 $\left(-K_\pi\right)^3\cdot H^3=0$,\,
 $\left(-K_\pi\right)^4\cdot H^2=27\delta$,\,
 $\left(-K_\pi\right)^5\cdot H=9\gamma $,\,
 $\left(-K_\pi\right)^6=81\beta$.

\end{enumerate}
Hence we obtain
\begin{align}
&\left(-K_\pi+ \tau H\right)^4 \cdot H^2= 27 \left( 4 \tau ^2 + \delta \right), \label{N^4}\\
&\left(-K_\pi+ \tau H\right)^5 \cdot H=9 \left( 20\tau ^3+15\delta\tau + \gamma \right), \label{N^5}\\
&\left(-K_\pi+ \tau H\right)^6=27 \left( 10\tau ^4+15\delta \tau ^2 + 2\gamma \tau + 3\beta \right). \label{N^6}
\end{align}

Since $\dim Z \leq 4$, we have two equalities
$\left(-K_\pi+ \tau H\right)^5.H=0$ and
$\left(-K_\pi+ \tau H\right)^6=0$,
that is,
\begin{eqnarray}
20\tau ^3+15\delta \tau + \gamma = 0, \label{A=0} \\
10\tau ^4+15\delta \tau ^2 + 2\gamma \tau + 3\beta = 0. \label{B=0} 
\end{eqnarray}
By \eqref{A=0}, \eqref{B=0} and the definition of $\beta$, we have 
\begin{align}
\delta = 5\,\tau ^2 \pm 3 \sqrt{ 5\tau ^4 - (a-b)^2}. \label{delta=}
\end{align}

\medskip
By Corollary~\ref{tau0}, it suffices to show that $\tau =0$.
Assume by contradiction that $\tau \neq 0$ in the sequel.

\begin{lemma}\label{taunum}
The following hold:
\begin{enumerate}
\item $\tau = 1$, $2$ or $3$, \label{taunum1}
\item $\left|a-b\right| = \tau ^2$ or $2\tau ^2$, \label{taunum2}
\item $\delta = 5 \tau ^2 \pm \dfrac{6\tau ^4}{\left|a-b\right|}$.
\label{taunum3}
\end{enumerate}
\end{lemma}

\begin{proof}
\ref{taunum1} By Proposition~\ref{tau}, it is enough to see that $\tau\in\bZ$.
Note that $\tau$ is a solution of the equation \eqref{B=0}.
Since $\delta$, $\gamma$, $\beta \in \bZ  $ and $\tau \in \bQ $ (by Proposition~\ref{tau}), we can write $\tau = m/10$ with $m\in \bZ $.
Then, by \eqref{B=0},
 $m^4 = -10 \left(15\delta m^{2}+20 \gamma m + 300 \beta\right)$.
So we have $m \equiv 0 \mod 10$.

\ref{taunum2} By \eqref{delta=}, there exists an integer $k$ such that $5\tau ^4 - \left(a-b \right)^2 = k^{2}$.
Therefore, for each $\tau=1$, $2$ or $3$, we have \ref{taunum2}.

Now, the assertion \ref{taunum3} follows from \eqref{delta=}.
\end{proof}

In any case, we have $\left( -K_\pi + \tau H \right)^4 \cdot H^2=27 \left(4 \tau ^2 + \delta \right) >0$.
Hence $\dim Z = 4$. 
Therefore, $Z \simeq \bP ^4$ or $\bQ ^4$ and $p$ is a $\bP ^{2}$-bundle by classification of CP surfaces and $4$-folds.
If $Z \simeq \bP ^4$, then we may apply Theorem~\ref{triv} for $p \colon X \to Z$, which we have already shown.
Then we have $X \simeq \bP ^2 \times \bP ^4$, a contradiction.
Hence $Z \simeq \bQ^{4}$:
\[\xymatrix{
                           &    X \ar[ld]_{\pi}    \ar[rd]^{p}     &                            \\
Y \simeq \bQ^4  &                                                     &  Z \simeq \bQ^4.  \\
}\]

There exists a rank $3$ vector bundle $\sF $ over $Z \simeq \bQ ^4$ such that $X \simeq \bP (\sF)$. We may assume that $c_{1}(\sF)=1$, $2$ or 3.
We will denote by $\eta$  the class of  tautological bundle on $\bP (\sF )$.

\begin{lemma}\label{ch}
The following hold:
\begin{enumerate}
\item  $\tau = c_{1}$.
In particular, $\sE $ is nef but not ample. \label{ch1}

\item
$c_2(\sE )=\left( {c_1}^2,0 \right) \text{ or } \left( 0,{c_1}^2 \right)$, and
$c_3=0$. \label{ch2}
\item
$27 \left(4 \tau ^2 + \delta \right) = 3^4 {c_1}^2$.

\label{ch3}
\end{enumerate}
\end{lemma}

\begin{proof}
\ref{ch1}
Note that $\sE $ is nef but not ample if and only if $\tau = c_1$.

Assume to the contrary that $\tau \neq c_1$.
Then, since $\left| \tau -c_1 \right| = 1$ or $2$, $-K_\pi+ \tau H$ is not a multiple of another divisor. 
Hence we have
\begin{align}
 -K_\pi+ \tau H= p^* H_Z, \label{primitive}
\end{align}
where $H_Z$ is the   ample generator of $\Pic(Z$). 

Since $-K_{Z} = 4 H_{Z}$,
we have $p^{*}\left(-K_{Z}\right) = 12 \xi + \left(4 \tau - 4c_{1} \right)H$. 
Hence, we have
\[
-K_{p}= -K_{X} - p^{*} \left( -K_{Z} \right) = -3 p^{*} H_{Z} + \left( 4  -  \tau \right)H.
\]
Therefore the slope $\tau _{Z}$ for the pair $(Z,\sF)$ is $3$.

Then, by Lemma~\ref{taunum}~\ref{taunum2} and \ref{taunum3} for $(Z,\sF)$, we have
\[
\delta_{Z} \coloneqq \left(c_{1}(\sF)^{2}-3c_{2}(\sF) \right) \cdot {H_{Z}}^{2} \equiv 0 \mod3.
\]
Hence $c_{1}(\sF) = 3$.

Since $\tau _{Z} = c_{1}(\sF)$, $\eta$ is nef but not ample.
Hence $\eta = H$.
By \eqref{primitive},
\[
3 \xi = \left( c_{1}-\tau \right) \eta + p^{*} H_{Z}.
\]
This contradicts the fact that $(\eta , p^{*} H_{Z})$ is a $\bZ$-basis of $\Pic (X)$.

\ref{ch2}
By \ref {ch1}, we have $a \geq 0 $, $b \geq 0$ and $\tau = c_1$.
Then, by Lemma~\ref{taunum} and the definition of $\delta$, we have the first assertion.
The second assertion follows from the equation~\eqref{A=0}.

\ref{ch3}
The assertion follows since $\delta = -\tau ^2$.
\end{proof}

Then, by Lemma~\ref{ch}~\ref{ch1} and the symmetry of $\pi \colon  X \to Y$ and $p \colon X \to Z$,  we have $\xi = p^* H_Z$, $\eta = H$ and $-K_\pi + \tau H = 3 \xi$.
Hence
\[
\left( p^* H_Z \right) ^4 \cdot \eta ^2 =
\xi ^4 \cdot H^2 =
\frac{\left( -K_\pi + \tau H \right) ^4 \cdot H^2}{3^4} \overset{\text{by }\eqref{N^4}}=
    \frac{27 \left( 4 \tau ^2 + \delta \right)}{3^4}  =
     c_{1}^{2}.
\]
The last equation follows from Lemma~\ref{ch}~\ref{ch3}.

By Lemma~\ref{taunum}~\ref{taunum1}, we have $c_1 ^2 = 1$, $4$ or $9$.
On the other hand, since $p$ is a $\bP ^2$-bundle over $\bQ^4$, we have $\left( p^* H_Z \right)^4.\eta ^2 = 2$.
This gives a contradiction.
Hence $\tau =0$, completing the proof. 
\end{proof}

\section{Products of CP manifolds}
In this section, we prove Proposition~\ref{pPb} below, which we use to prove that a certain CP manifold is a product of CP manifolds.

\begin{proposition}\label{pPb}
Let $X$ be a smooth projective variety
and suppose that, for some integer $r \geq 2$, there exist   a smooth contraction $f \colon X \to Y$ of relative dimension $r-1$ and   another contraction $g \colon X \to Z$ onto an  $r-1$-dimensional manifold $Z$.

Assume that $g$ does not contract any curve contained in an $f$-fiber, and assume moreover that one of the following holds:
\begin{enumerate}
\item $f$ is a smooth $\bP ^{r-1}$-fibration, or
\item every $f$-fiber is a smooth hyperquadric of dimension $r-1 \geq 3$.
\end{enumerate}
Then $-K_{f}$ is nef.
\end{proposition}

\begin{remark}
If $f$ is a \emph{$\bP ^{r-1}$-bundle} and $X$ is a Fano manifold, then the proposition follows from \cite[Lemma~4.1]{NO}.
\end{remark}

\begin{proof}
We may assume that $Y$ is a smooth projective curve.
Then $\rho_{Y}=2$ in any case.
We introduce an invariant $\tau$ which satisfies $-K_f + \tau F$ is nef but not ample, where $F$ is the numerical equivalence class of an $f$-fiber (cf. Definition~\ref{slope}).
Then $g$ is defined by the semiample divisor $-K_f + \tau F$.

\medskip
First we treat the case where $f$ is a smooth $\bP ^{r-1}$-fibration.
In this case $f$ is a smooth $\bP ^{r-1}$-bundle (see Remark~\ref{rem_bundle}~\ref{rem_bundle2}).
Hence, we have $X \simeq \bP(\sE)$ for some vector bundle $\sE$ over $Y$. 
We will denote by $F \in N^1(X)$ the class of a fiber and by $\xi \in N^1(X)$ the class of the tautological divisor.

Since $\dim Z = r-1$, we have $\left( -K_f+ \tau F \right)^{r}=0$.
Note that $\xi ^r = \deg(\det \sE)$ and $-K_f \nequiv r \xi - \deg(\det \sE)F$.
Hence we have $\tau =0$, namely, $-K_{f}$ is nef.

\medskip

Next, we  treat the case where every $f$-fiber is a smooth quadric.
We will denote by $F \in N^1(X)$ the class of an $f$-fiber.

By \cite[Proposition~21]{Ar2},
there exists a triple $(\sE,\sL,s)$ which satisfies the following:
\begin{enumerate}
\item $\sE $ (resp. $\sL $) is a vector bundle of rank $r+1$ (resp.\ a line bundle)  over $Y$.
\item $q \in H^0 (S^2 \sE  \otimes \sL )$.
\item $X$ is a zero scheme of $q$ in $\bP (\sE )$
\end{enumerate}
Set $d \coloneqq \deg(\det\sE)$ and $\ell \coloneqq \deg\sL$.
Since $f$ is smooth, we have 
\begin{align}
 -2d = \left(r+1\right)\ell. \label{deg}
\end{align}

By adjunction, we have $-K_{f} \nequiv \left(r-1\right)\xi_X - \left(d+\ell \right) F$, where $\xi _X$ is the restriction of the tautological divisor $\xi$ on $\bP (\sE )$.

Since $\dim Z = r-1$, we have $\left( -K_f+ \tau F \right) ^{r}=0$,
that is, 
\begin{align}
 \bigl( \left( r-1 \right) \xi_X \bigr) ^{r} + r\left( \tau -d- \ell \right)  \bigl( \left(r-1\right)\xi _X \bigr)^{r-1} \cdot F =0. \label{quad}
\end{align}
Note that $X \nequiv 2\xi + \ell F' $  in $N^1(\bP(\sE)) $ and $\xi ^{r+1} = d$, where $F'$ is the class of a fiber of $\bP(\sE) \to Y$.
Hence  we have ${\xi _X}^{r} =2d+\ell$ and ${\xi _X}^{r-1}  \cdot F =2$. 
Therefore, by \eqref{deg} and \eqref{quad}, we have $\tau = 0$.
\end{proof}

\section{CP Manifolds with Large Picard Number}\label{large-rho}
In this section, we prove Theorems~\ref{k=2} and \ref{k=4}, which will complete our proof of Theorem~\ref{rhon-4}.
Theorem~\ref{k=2} was obtained independently by K. Watanabe.
See also \cite[Proposition~2.4]{BCDD}, \cite[Proposition~5.1]{NO} or \cite[Proposition~2.3]{W2} for the case $n - \rho _{X} = 0$ or $1$.
We include our proof of them for completeness of our treatment.

First we prove the following:
\begin{theorem}\label{2rho}
Let $X$ be a CP $n$-fold which does not admit a contraction onto an FT manifold.
Then  $n \geq 2\rho_{X}$.
Furthermore, the following hold:
\begin{enumerate}

\item
If the equality holds, then $X \simeq \pbP{2}{\rho_X}$ \label{2rho1}.

\item
If $ n = 2\rho_X+1$, then
$X \simeq \pbP{2}{\rho_X -1} \times \bP ^3$,
$\pbP{2}{\rho_X -1} \times \bQ ^3$,
$\pbP{2}{\rho_X -2} \times \bP (\sS _i)$ or
$\pbP{2}{\rho_X -2} \times \bP (T_{\bP ^3} )$. \label{2rho2}

\end{enumerate}
\end{theorem}

\begin{proof}
First, we prove  by induction on $n$ that every CP $n$-fold
with $2 \rho _X > n$ admits a contraction onto an FT manifold.
The case $n=1$ is trivial.

Assume $n>1$.
By Proposition~\ref{cont}, we have a sequence of smooth elementary contractions
\begin{align}\label{sequence}
X=X_0 \to X_1 \to \cdots  \to X_{2 \rho_X -n}  \cdots \to X_{\rho_X -1} \to X_{\rho_X}=\text{point},
\end{align}
where each $X_i$ is a CP manifold of dimension $\leq n-i$ with Picard number $ \rho _X -i$.

If $\dim X_{2\rho_X -n} = 2n-2\rho_X$ for every sequence \eqref{sequence}, then $X _{2\rho_X-n-1}$ is an FT manifold.

Otherwise $\dim X_{2\rho_X -n} < 2n-2\rho_X$
for some sequence \eqref{sequence}.
Then we have 
\[
\dim X _{2\rho_X -n}< 2n - 2 \rho _X = 2 \rho _{ X _{2\rho_X -n}}.
\]
Thus by inductive hypothesis $ X _{2\rho_X -n}$ admits a contraction onto an FT manifold, and then so does $X$.

\medskip
Next, we prove \ref{2rho1} and \ref{2rho2}.

\ref{2rho1} we proceed by induction on $n$.
If $n=2$, then $X\simeq \bP^2$ and the assertion holds.
Hence we assume $n>2$.
Then, by Proposition~\ref{cont}, there exists a smooth elementary contraction $f \colon X \to Y$.

If $\dim Y < n-2$, then $2 \rho_Y > \dim Y$.
Hence $Y$ is admits a contraction onto an FT manifold. This contradicts our hypothesis.
Hence $\dim Y \geq n-2$  for every elementary contraction $f \colon X \to Y$.
Furthermore, since  $X$ is not an FT manifold,  there exists an elementary contraction $f \colon X \to Y$ with $\dim Y = n-2$.
Then, by inductive hypothesis, $Y \simeq  \pbP{2}{\rho_Y}$.
Hence $f$ is a $\bP ^2$-bundle.
Furthermore $f$ is  trivial on each factor $\bP^{2}$ of $Y$ by Theorem~\ref{CP5}.
Hence $X \simeq  \pbP{2}{\rho_X}$ by Proposition~\ref{trivb} \ref{trivb5} $\Rightarrow$ \ref{trivb1}.

\medskip
\ref{2rho2}
We proceed by induction on $n$.
If $n=3$ or $5$, the assertion follows from Theorem~\ref{CP5}.
Hence we assume that $n>5$.
By our hypothesis, there exists an elementary contraction $f \colon X \to Y$ with $n-3 \leq \dim Y \leq n-2$.

If $\dim Y = n-3$, then $Y \simeq \pbP{2}{\rho_X -1}$ by \ref{2rho1}.
Let  $g \colon X \to Z$ be the elementary contraction such that the following diagram is commutative:

\[\xymatrix{
         X          \ar[rr]^g    \ar[d]_f  &  &    Z \ar[d]                            \\
Y \simeq \bP ^2 \times \pbP{2}{\rho_X -2}     \ar[rr]^{\pr _2} \ar[d]_{\pr_{1}} &  &    \pbP{2}{\rho_X -2} \ar[d]\\
\bP ^2 \ar[rr] && \text{point}
}\]
Then, by the classification of CP $5$-fold with Picard number two, every fiber of $\pr _2 \circ f$ is isomorphic to  $\bP^{2}\times \bP^{3}$ or $\bP ^{2} \times \bQ^{3}$.
Thus $g$ is a $\bP ^2$-fibration.

Hence we may find a elementary contraction $f \colon X \to W$ with $\dim W = n-2$.
Then $W \simeq  \pbP{2}{\rho_X -2} \times \bP ^3$,
                                                     $\pbP{2}{\rho_X -2} \times \bQ ^3$,
                                                     $\pbP{2}{\rho_X -3} \times \bP (\sS _i)$,
                                                     $\pbP{2}{\rho_X -3} \times \bP (T_{\bP ^3} )$ by inductive hypothesis.
In any case $f$ is a $\bP ^{2}$-bundle.
Furthermore, if the last three cases occur, then $f$ is trivial on any fiber of the elementary contractions of $W$ by the classification of CP $m$-folds with $m=3$, $4$ or $5$.
Hence $X \simeq  \bP^{2} \times W$ by Proposition~\ref{trivb} \ref{trivb5} $\Rightarrow$ \ref{trivb1}.

Hence we may assume $W \simeq  \pbP{2}{\rho_X -2} \times \bP ^3$.
Let  $g \colon X \to V$ be the elementary contraction such that the following diagram is commutative:
\[\xymatrix{
X \ar[rr]^g \ar[d]_{f} && V \ar[d] \\
W \simeq  \bP^2 \times \left(\pbP{2}{\rho_X -3} \times \bP ^3\right) \ar[d]_{\pr _1} \ar[rr]^{\pr _2} &&  \left(\pbP{2}{\rho_X -3} \times \bP ^3 \right) \ar[d]\\
\bP ^2 \ar[rr] && \text{point}
}\]
Then by the classification of CP $4$-fold, every fiber of $\pr _2 \circ f$ is isomorphic to $\bP^2 \times \bP^2$.
Thus $g$ is a $\bP ^2$ -fibration.
By inductive hypothesis, we have $V \simeq \pbP{2}{\rho_X -2} \times \bP ^3$,
                                                     $\pbP{2}{\rho_X -2} \times \bQ ^3$,
                                                     $\pbP{2}{\rho_X -3} \times \bP (\sS _i)$,
                                                     $\pbP{2}{\rho_X -3} \times \bP (T_{\bP ^3} )$.
Hence $g$ is a $\bP^2$-bundle.
Since every $g$-fiber is not contracted by $\pr _1 \circ f$, it follows from Proposition~\ref{pPb} that $-K_g$ is nef.
Hence the assertion follows from Proposition~\ref{trivb} \ref{trivb2} $\Rightarrow$ \ref{trivb1}.
\end{proof}

\begin{theorem}\label{k=2}
Let $X$ be a CP $n$-fold with $n-\rho_{X} \leq 3$.
Then $X$ is a rational homogeneous manifold.
\end{theorem} 

\begin{proof}
We may assume that $ 2 \rho_{X} +2 \leq n$ by Corollary~\ref{FTdom}, Theorem~\ref{2rho} and induction on $n$.
Then, by the inequality $2 \rho_{X} +2 \leq n \leq  \rho _{X} +3$, we have $n \leq 4$ and the assertion follows from Theorems~\ref{CP5}.
\end{proof}

Finally, we prove the following:
\begin{theorem}\label{k=4}
Let $X$ be a CP $n$-fold with $n-\rho_{X} = 4$.
Then $X$ is a rational homogeneous manifold.
\end{theorem}

\begin{proof}
By Corollary~\ref{FTdom}, Theorems~\ref{2rho} and \ref{k=2} and induction on $n$,
we may assume that $ 2 \rho_{X} +2 \leq n$, and hence $n \leq 6$.
The case $n=5$ follows from Theorem~\ref{CP5}.

Assume that $n=6$.
Then $\rho_{X}=2$, and there are two different smooth elementary contractions $f \colon X \to Y$ and $g \colon X \to Z$.
 Without loss of generality, we may assume that $ \dim Y \geq \dim Z $.
  Furthermore, $\dim Y \geq 3$ by the inequality $ \dim X \leq \dim Y + \dim Z$.

If $ \dim Y = 5$, then $Y$ is isomorphic to $\bP ^{5}$, $\bQ ^{5}$ or $K(G_{2})$ by Theorem~\ref{CP5}.
Since $Y$ is rational, $f$ is a $\bP^1$-bundle.
Hence $X$ is a rational homogeneous manifold by Proposition~\ref{P1b}.

If $ \dim Y = 4$, then $Y$ is isomorphic to $\bP ^4$ or $ \bQ ^4 $ by Theorem~\ref{CP5}.
Since $Y$ is rational, $f$ is a $\bP^2$-bundle.
Hence $X$ is isomorphic to $\bP ^2\times\bP ^4$ or $\bP ^2\times\bQ ^4$ by Theorem~\ref{triv}.

In the remaining case,
we have $ \dim Y = \dim Z = 3$.
Hence $Y \simeq \bP ^{3}$ or $\bQ ^{3}$, and $Z \simeq \bP ^{3}$ or $\bQ ^{3}$ by Theorem~\ref{CP5}.
Then, by Proposition~\ref{pPb}, $-K_{f}$ and $-K_{g}$ are nef.
Hence we have $-K_X = - {f}^* K_{Y} - {g}^* K_{Z} $.
Then, by purity of branch locus, $(f,g): X \to Y \times Z$ is \'etale.  
Since $Y \times Z$ is simply connected, we have $X \simeq Y \times Z$.
This completes the proof.
\end{proof}

\bibliographystyle{amsplain}
\bibliography{}

\end{document}